\documentclass[12pt,a4paper]{amsart}

\usepackage{amsfonts}
\usepackage{amssymb}
\usepackage{amsthm}
\usepackage{amsmath}
\usepackage{enumerate}
\usepackage{hyperref}
\usepackage{amsrefs}
\usepackage{graphicx}

\newcommand{\IC}{\mathbb{C}}

\newtheorem{proposition}{Proposition}[section]
\newtheorem{lemma}[proposition]{Lemma}
\newtheorem{theorem}[proposition]{Theorem}
\newtheorem{corollary}[proposition]{Corollary}
\theoremstyle{definition}

\theoremstyle{remark}

\newtheorem{remark}[proposition]{Remark}

\numberwithin{equation}{section}

\begin{document}

\title{Universal enveloping TROs and structure of W$^*$-TROs}

\author{Bernard Russo}
\address{Department of Mathematics, University of California, Irvine, CA 92697-3875, USA}
\email{brusso@math.uci.edu}

\date{\today}

\begin{abstract}
Calculation of  the universal enveloping TROs of continuous JBW$^*$-triples, and application of  the techniques used to supplement  the structural results of Ruan for W$^*$-TROs.
\end{abstract}

\maketitle
\section{Introduction}

In 2004, Ruan \cite{Ruan04} presented a classification scheme and proved various structure theorems for weakly closed ternary rings of operators (W$^*$-TROs) of particular types.  A W$^*$-TRO of type I, II, or III was defined according to the Murray-von Neumann type of its linking von Neumann algebra.  
W$^*$-TROs of type II were further designated as either of type $II_{1,1},II_{1,\infty}, II_{\infty,1}$ or     $II_{\infty,\infty}$.  Representation theorems for W$^*$-TROs of various types were given in Ruan's paper (see Theorem~\ref{thm:0805161} below), but with the possible exception of type $II_{1,1}$ (however, see the end of subsection~\ref{subs:2.1}).

  The purpose of this paper is  to shed some light on the structure of W*-TROs (Proposition~\ref{thm:0608161}), and in particular, those of type $II_{1,1}$ (Corollary~\ref{cor:0718161}), by using ideas from \cite{BunTimJLMS13},  together with the well established structure theory of JBW*-triples (cf. \cite{Horn87,Horn88}). 
A W$^*$-TRO is an example of a JBW$^*$-triple.

 Let us recall   the structure of all JBW*-triples $U$:   there is a surjective linear isometric triple isomorphism
\begin{equation}\label{eq:0718161}
U\mapsto \oplus_\alpha L^\infty(\Omega_\alpha,C_\alpha)\oplus pM\oplus H(N,\beta),
\end{equation}
where each $C_\alpha$ is a Cartan factor, $M$ and $N$ are continuous von Neumann algebras, $p$ is a projection in $M$, and  $\beta$ is a *-antiautomorphism of $N$ of order 2 with fixed points $H(N,\beta)$.

A basic tool in our approach is the universal enveloping TRO $T^*(X)$ of a JC$^*$-triple $X$ as developed in \cite{BunFeeTim12} and its sequels \cite{BunTim13,BunTimJLMS13}.
By \cite[Theorem 4.9]{BunTim13},  $$T^*(L^\infty(\Omega)\otimes C)=L^\infty(\Omega)\otimes T^*(C),$$  and consequently (see Proposition~\ref{prop:0806161} below),
 identifying $L^\infty(\Omega,C)$ with $L^\infty(\Omega)\otimes C$,
 \begin{equation}\label{eq:0608161}
T^*(X)\simeq \oplus_\alpha L^\infty(\Omega_\alpha,T^*(C_\alpha))\oplus T^*(pM)\oplus T^*(H(N,\beta)).
\end{equation}

The TROs $T^*(C)$ where $C$ is a Cartan factor have been determined in \cite{BunFeeTim12}, and independently and simultaneously in the finite dimensional cases in \cite{BohWer14}.  Both \cite{BunFeeTim12} and \cite{BohWer14} make very strong use of \cite{NeaRus03}. 

Our main new results are the determination of the TROs $T^*(pM)$ and $T^*(H(N,\beta))$.  In 
 Theorem~\ref{prop:0416161} it is shown that 
$T^*(pM)=pM\oplus M^tp^t$, and in Theorem~\ref{lem:0402161},  that $T^*(H(N,\beta))=N$.  

Only one of these results is needed in the proof of Proposition~\ref{thm:0608161} but each is of interest in its own right.  In addition, alternate proofs of portions of 
Proposition~\ref{thm:0608161}, which use both of these results, are provided in section~\ref{sect5} as an illustration of the power of universal enveloping TROs.  It is planned to use this technique in future research.

A representation result, obtained simultaneously and independently by different methods in 2013, and stated in the following theorem, plays a key role in some of our proofs, .  


\begin{theorem}\label{thm:0802161}

{\rm (a)  (Bunce-Timoney \cite[Lemma 5.17]{BunTimJLMS13})} A W$^*$-TRO is TRO-isomorphic to the direct sum $eW\oplus Wf$, where $W$ is a von Neumann algebra and $e,f$ are centrally orthogonal projections in $W$.\smallskip

{\rm (b) (Kaneda \cite[Theorem]{KanedaPJM13})} A W$^*$-TRO $X$ can be decomposed into the direct sum of  TROs $X_L,X_R,X_T$, and there is a complete isometry of $X$ into a von Neumann algebra $M$ which maps $X_L$ (resp. $X_R,X_T$)  into a weak*-closed left ideal (resp. right ideal, two-sided ideal)
\end{theorem}

\section{Preliminaries}

A ternary ring of operators (hereafter TRO) is a norm closed complex subspace of $B(K,H)$ which contains $xy^*z$ whenever it contains $x,y,z$, where $K$ and $H$ are complex Hilbert spaces. A TRO which is closed in the weak operator topology is called a W$^*$-TRO. A TRO-homomorphism is a linear map $\varphi$ between two TROs respecting the ternary product: $\varphi(xy^*z)=\varphi(x)\varphi(y)^*\varphi(z)$.

The definition of JB$^*$-triple will not be given here (see for example \cite{chu,BunFeeTim12,Horn87,Horn88}), since only its concrete realizations, which are called JC$^*$-triples, will be involved,  namely, norm closed complex subspaces of $B(K,H)$ which contain $xy^*z+zy^*x$ whenever they contain $x,y,z$.  A JC$^*$-homomorphism is a linear map $\varphi$ between two JC$^*$-triples respecting the triple product: $\{x,y,z\}:=(xy^*z+zy^*x)/2$, that is, $\varphi\{x,y,z\}=\{\varphi(x),\varphi(y),\varphi(z)\}$. Such maps are called triple homomorphisms to distinguish them from TRO-homomorphisms.

A JC-algebra is a norm closed real subspace of $B(H)$ which is stable for the Jordan product $x\circ y=(xy+yx)/2$.  A JC$^*$-algebra is a norm closed complex  Jordan *-subalgebra of $B(H)$.

Corresponding to an orthonormal basis of a complex Hilbert space $H$, let $J$ be the unique conjugate linear isometry which fixes that basis elementwise.  The transpose $x^t\in B(H)$ of an element $x\in B(H)$ is then defined by $x^t=Jx^*J$ 

\subsection{Ruan Classification Scheme}\label{subs:2.1}

If $R$ is a von Neumann algebra and $e$ is a projection in $R$, then $V:=eR(1-e)$ is a W*-TRO.  Conversely if $V\subset B(K,H)$ is a W*-TRO, then with $V^*=\{x^*:x\in V\}\subset B(H,K)$, $M(V)=\overline{XX^*}^{sot} \subset B(H)$, $N(V)=\overline{X^*X}^{sot} \subset B(K)$, let 
\[
R_V=\left[\begin{matrix}
M(V)&V\\
V^*&N(V)
\end{matrix}\right]\subset B(H\oplus K)
\]
denote the  linking von Neumann algebra of $V$.  Then there is a SOT-continuous TRO-isomorphism $V\simeq eRe^\perp$, where $e=\left[\begin{smallmatrix} 1_H&0\\ 0&0\end{smallmatrix}\right]$ and $e^\perp=\left[\begin{smallmatrix} 0&0\\ 0&1_K\end{smallmatrix}\right]$.\smallskip

In particular, if $V=pM$ where $p$ is a projection in a von Neumann algebra $M$, then
\[
R_V=\left[\begin{matrix}
pMp&pM\\
Mp&c(p)M
\end{matrix}\right]\subset B(H\oplus H),
\]
where $c(p)$ denotes the central support of $p$ (see \cite[p.\ 965]{BunFeeTim12}).

A W$^*$-TRO $V$ is of type I,II, or III according as $R_V$ is a von Neumann algebra of the corresponding type. 
A W$^*$-TRO of type II is said to be of type $II_{\epsilon,\delta}$, where $\epsilon,\delta \in \{1,\infty\}$, if $M(V)$ is of type 
$II_{\epsilon}$ and $N(V)$ is of type $II_{\delta}$.

Ruan's main representation theorems from \cite{Ruan04} are summarized in the following theorem.
\begin{theorem}\label{thm:0805161}\rm{(}{\bf Ruan \cite{Ruan04}\rm{)}}  Let $V$ be a W$^*$-TRO.
\begin{description}
\item[i]
If $V$ is a  W*-TRO of  type I, then $V$ is TRO-isomorphic to $\oplus_\alpha L^\infty(\Omega_\alpha, B(K_\alpha,H_\alpha))$.  \rm{(}\cite[Theorem 4.1]{Ruan04}\rm{)}
\item[ii]
If $V$ is a  W*-TRO of one of the types $I_{\infty,\infty}, II_{\infty,\infty}$ or $III$, acting on a separable Hilbert space, then $V$ is a stable W*-TRO, and hence TRO-isomorphic to a von Neumann algebra.
\rm{(}\cite[Corollary 4.3]{Ruan04}\rm{)}
\item[iii]
If $V$ is a  W*-TRO of  type $II_{1,\infty}$ (respectively $II_{\infty,1}$), then $V$ is  TRO-isomorphic to $B(H,\IC)\otimes M$ (respectively $B(\IC,H)\otimes N$), where $M$ (respectively $N$) is a von Neumann algebra of type $II_1$.   (\cite[Theorem 4.4]{Ruan04})
\end{description}
\end{theorem} 

According to Ruan \cite[page 862]{Ruan04}, ``The structure of a type $II_{1,1}$ W*-TRO is a little bit more complicated.''    Nevertheless, using techniques developed for approximately finite dimensional (AFD) von Neumann algebras of type $II_1$, he is able to prove (\cite[Theorem 5.4]{Ruan04}) that every injective  W*-TRO of type $II_{1,1}$ acting on a separable Hilbert space is rectangularly AFD
(approximately finite dimensional). Together with other results from \cite[Sections 3,4]{Ruan04}, he proves that any W*-TRO acting on a separable Hilbert space is injective if and only if it is rectangularly AFD (\cite[Theorem 5.5]{Ruan04}).

\subsection{Horn-Neher Classification Scheme}\

A complex \emph{JBW$^*$-triple} is a complex JB$^*$-triple which is also a dual Banach space.
The structure of JBW$^*$-triples is fairly well understood.  Every JBW$^*$-triple is a direct sum of a JBW$^*$-triple of type I and a continuous JBW$^*$-triple (defined below). JBW$^*$-triples of type I have been defined and classified in \cite{Horn87} and continuous JBW$^*$-triples have been classified in \cite{Horn88}.
JBW$^*$-triples of type I will not be defined here. Their classification theorem from \cite{Horn87} states:  A JBW$^*$-triple of type I is an $\ell^\infty$-direct sum of JBW$^*$-triples of the form $A\otimes C$, where $A$ is a commutative von Neumann algebra and $C$ is a Cartan factor. (For Cartan factors of types 1-6, see \cite[Theorem 2.5.9 and page 168]{chu}. A Cartan factor of type 1 is by definition  $B(H,K)$, where $H$ and $K$ are complex Hilbert spaces. No other  information about  Cartan factors is needed in this paper)\smallskip

A $JBW^\ast$-triple $\mathcal{A}$ is said to be
\textit{continuous} if it has no type I direct summand. In this
case it is known that, up to isometry, $\mathcal{A}$ is a $JW^\ast$-triple, that is, a subspace of the bounded operators on a Hilbert space which is closed under the triple product $xy^*z+zy^*x$ and closed in the weak operator topology. More importantly,  it has a
unique decomposition into weak$^*$-closed triple ideals, $\mathcal{A} = H(W,\alpha)\oplus pV,$ where $W$
and $V$ are continuous von Neumann algebras, $p$ is a projection
in $V$, $\alpha$ is   a $^*$-antiautomorphism of $W$ order 2 and
$H(W,\alpha)=\{x\in W: \alpha(x)=x\}$ (see \cite[(1.20) and section 4]{Horn88}).  Notice that the triple product in $pV$ is given by $(xy^*z+zy^*x)/2$ and that $H(W,\alpha)$ is a JBW$^*$-algebra with the Jordan product $x\circ y=(xy+yx)/2$.\smallskip


A continuous JBW$^*$-triple of the form $pM$ (which is a W$^*$-TRO),  is said to be of associative type, and is classified into four types in \cite{Horn88} as follows.

\begin{itemize}
\item $II_1^a$ if $M$ is of type $II_1$ and $p$ is (necessarily) finite.\smallskip
\item $II_{\infty,1}^a$ if $M$ is of type $II_\infty$ and $p$ is a finite projection.\smallskip
\item $II^a_\infty$ if $M$ is of type $II_\infty$ and $p$ is a properly infinite projection.\smallskip
\item $III^a$ if $M$ is of type III and $p$ is a (necessarily) properly infinite projection.
\end{itemize}

A continuous JBW$^*$-triple of the form $H(W,\alpha)$ (which is a JBW$^*$-algebra), is said to be of hermitian type, and  is classified into three types in \cite{Horn88} as follows.

\begin{itemize}
\item $II_1^{herm}$ if $W$ is of type $II_1$.\smallskip
\item $II_{\infty}^{herm}$ if $W$ is of type $II_\infty$.\smallskip
\item $III^{herm}$ if $W$ is of type III.
\end{itemize}



\subsection{Universal Enveloping TROs}

If $E$ is a JC*-triple, denote by $C^*(E)$ and $T^*(E)$ the universal C*-algebra and the universal TRO of $E$ respectively (see \cite[Theorem 3.1,Corollary 3.2, Definition 3.3]{BunFeeTim12}).  Recall that the former means that $C^*(E)$ is a C*-algebra,  there is an injective JC*-homomorphism $\alpha_E:\rightarrow C^*(E)$ with the properties that $\alpha_E(E)$ generates $C^*(E)$  as a C*-algebra and for each JC*-homomorphism $\pi:E\rightarrow A$, where $A$ is a C*-algebra, there is a unique 
*-homomorphism $\tilde\pi:C^*(E)\rightarrow A$ such that $\tilde\pi\circ\alpha_E=\pi$. 
The latter  means that $T^*(E)$ is a TRO, there is an injective TRO-homomorphism $\alpha_E:\rightarrow T^*(E)$ with the properties that $\alpha_E(E)$ generates $T^*(E)$  as a TRO and for each JC*-homomorphism $\pi:E\rightarrow T$, where $T$ is a TRO, there is a unique 
TRO-homomorphism $\tilde\pi:T^*(E)\rightarrow T$ such that $\tilde\pi\circ\alpha_E=\pi$.   

In several places in the papers \cite{BunFeeTim12,BunTim13,BunTimJLMS13}, reference is made to the fact that the universal TRO
construction commutes with finite direct sums of JC*-triples. More generally:

\begin{proposition}\label{prop:0806161}
If $E_i$ ($i\in I$) is a family of JC$^*$-triples, then $$T^*(\oplus_iE_i)=\oplus_iT^*(E_i).$$
\end{proposition}
\begin{proof}
 Let $E=\oplus_iE_i$.   It will be shown that  $(R,\beta):=(\oplus_iT^*(E_i),\oplus\alpha_{E_i})$ satisfies the properties enjoyed by $(T^*(E),\alpha_E)$, that is, $R$ is a TRO and $\beta:E\rightarrow R$ is an injective triple isomorphism such that \smallskip

(a) $\beta(E)$ generates $R$ as a TRO;\smallskip

(b) for each triple homomorphism $\pi:E\rightarrow T$, where $T$ is a TRO, there is a (necessarily unique) TRO homomorphism $\tilde\pi:R\rightarrow T$ such that $\tilde\pi\circ\beta=\pi$.

\medskip

It is clear that $R$ is a TRO, $\beta$ is an injective triple isomorphism, and $\beta(E)$ generates $R$ as a TRO.  Let $\pi:E\rightarrow R$ be a triple homomorphism.  Then $\pi_i:=\pi|E_i$ is a triple homomorphism from $E_i$ to $T$, so there exists a TRO homomorphism $\tilde\pi_i:T^*(E_i)\rightarrow T$ such that $\tilde\pi_i\circ\alpha_{E_i}=\pi_i$.\smallskip

  Consider the TRO homomorphism $\sigma:=\oplus_i\tilde\pi_i:R\rightarrow \oplus_i\pi_i(E_i)$.
Since the $E_i$ are pairwise orthogonal ideals in $E$, the $\pi(E_i)$ are pairwise orthogonal (triple) ideals in $T$ and $\oplus_i\pi_i(E_i)\subset T$, that is, $\sigma$ has range in $T$.  Moreover, it is easily verified that $\sigma\circ\beta=\pi$ so that  $\tilde\pi$ may be taken to be $\sigma$.
\end{proof}

The property of being universally reversible (cf. \cite{BunTimJLMS13}) will be important for our proofs.
A JC-algebra  $A\subset B(H)_{sa}$ is called {\it reversible} if 
\[
a_1,\ldots,a_n\in A\Rightarrow a_1\cdots a_n+a_n\cdots a_1\in A.
\]
$A$ is {\it universally reversible} if $\pi(A)$ is reversible for each representation (=Jordan homomorphism) $\pi:A\rightarrow B(K)_{sa}$.  
A JC$^*$-algebra  $A\subset B(H)$ is called {\it reversible} if 
\[
a_1,\ldots,a_n\in A\Rightarrow a_1\cdots a_n+a_n\cdots a_1\in A.
\]
and 
$A$ is {\it universally reversible} if $\pi(A)$ is reversible for each representation (=Jordan $^*$-homomorphism) $\pi:A\rightarrow B(K)$.  Since JC-algebras are exactly the self-adjoint parts of JC$^*$-algebras, a JC$^*$-algebra $A$ is reversible (respectively, universally reversible) if and only if the JC-algebra $A_{sa}$ is reversible (respectively, universally reversible).

A JC$^*$-triple  $A\subset B(H,K)$ is called {\it reversible} if $a_1,\ldots,a_{2n+1}\in A\Rightarrow$
\[
 a_1a_2^*a_3\cdots a_{2n-1}a_{2n}^*a_{2n+1}+a_{2n+1}a_{2n}^*a_{2n-1}\cdots a_3a_2^*a_1\in A.
\]
and 
$A$ is {\it universally reversible} if $\pi(A)$ is reversible for each representation (=triple homomorphism) $\pi:A\rightarrow B(H',K')$.  
It is easy to check that if a JC$^*$-algebra is universally reversible as a JC$^*$-triple, then it is universally reversible as a JC$^*$-algebra.






Given a JC-algebra $A$ , there is a universal C$^*$-algebra $B$ of $A$, analogous to the definition of $C^*(E)$ given above for JC$^*$-triples $E$, with the following properties: there is a Jordan homomorphism $\pi$ from $A$ into $B_{sa}$ such that $B$ is the C$^*$-algebra generated by $\pi(A)$ and for every Jordan homomorphism $\pi_1$ from $A$ into $C_{sa}$ for some C$^*$-algebra $C$, there is a $^*$-homomorphism $\pi_2:B\rightarrow C$ such that $\pi_1=\pi_2\circ\pi$. (see \cite[section 4]{HancheOlsen03} or \cite[Proposition 4.36]{AlfSchbook}).  It is clear that $B=C^*(E)$ where $E$ is the complexification of $A$.

For the convenience of the reader,  the following theorem is stated.

\begin{theorem}\label{thm:0403161}\rm{(}\cite[Theorem 4.4]{HancheOlsen03}\rm{)}
Let $A$ be a universally reversible JC-algebra, $B$  a C$^*$-algebra , and $\theta:A\rightarrow B_{sa}$ an injective homomorphism such that $B$ is the C$^*$-algebra generated by $\theta(A)$.
If $B$ admits an antiautomorphism $\varphi$ such that $\varphi\circ\theta=\theta$, then $\theta$ extends to a $^*$-isomorphism  of  $C^*(A)$ onto $B$.
\end{theorem}


\section{The universal enveloping TROs of $pM$ and of $H(N,\beta)$}\label{sect4}

The proofs of the theorems in this section are very short since several results from \cite{BunTimJLMS13} are used, as well as one each from \cite{BunFeeTim12} and \cite{Gasemyr90}.

\subsection{The universal enveloping TRO of $pM$}

\begin{lemma}\label{lem:0416161}
Let $W$ be a continuous von Neumann algebra, and let $e$ be a projection in $W$.  Then the TRO
$eW$ does not admit a nonzero TRO homomorphism onto $\IC$.
\end{lemma}
\begin{proof}
Suppose, by way of contradiction, that $f$ is a nonzero TRO homomorphism of $eW$ onto $\IC$.  
Since $f(e)=f(ee^*e)=f(e)|f(e)|^2$, either $f(e)=0$ or $|f(e)|=1$.  The former case can be ruled out  since for $x\in W$, 
$f(ex)=f((e1)(e1)^*(ex))=|f(e)|^2f(ex)$ and $f$ would be zero.  If then $f(e)=\lambda$ with $|\lambda|=1$, then replacing $f$ by $\overline\lambda f$ it can be assumed that $f(e)=1$.

For $x,y\in W$, $f((exe)(eye))=f(exee^*eye)=f(exe)\overline{f(e)}f(eye)=f(eye)f(eye)$ and 
$f((exe)^*)=f(ex^*e)=f(e(exe)^*e)=\overline{f(exe)}$ so that $f|eWe$ is a $^*$-homomorphism onto $\IC$ and since $f(e)=1=\|f\|$, $f|eWe$ is a state of $eWe$.  Moreover $f|eWe$, being a $^*$-homomorphism is order preserving and has the value 0 or 1 on each projection of $eWe$.  It follows trivially that $f$ is completely additive on projections and is therefore a normal functional by a theorem of Dixmier \cite[1.13.2, and page 30]{sakai} or \cite{Dixmier53}.  Now apply the theorem of Plymen (\cite{Plymen68}) to the effect that a continuous von Neumann algebra admits no dispersion-free normal state.  (A state is dispersion-free if it preserves squares of self-adjoint elements.)
\end{proof}
\begin{theorem}\label{prop:0416161}
Let $W\subset B(H)$ be a continuous von Neumann algebra, and let $e$ be a projection in $W$.  Then 
$T^*(eW)=eW\oplus W^te^t$, where $x^t$ be any transposition on $B(H)$.
\end{theorem}

\begin{proof}  By \cite[Proposition 3.9]{BunTimJLMS13}, $eW$ is universally reversible and so by \cite[Theorem 4.11]{BunTimJLMS13}, it does not admit a TRO homomorphism onto a Hilbert space of dimension greater  than 2. The proof is completed by applying Lemma~\ref{lem:0416161} and  \cite[Theorem 5.4]{BunTimJLMS13}.
\end{proof}

\subsection{The universal enveloping TRO of $H(N,\beta)$}


Let $E$ be a JC$^*$-algebra.  Similar to the construction of $C^*(E)$ when $E$ is considered as a JC$^*$-triple, there is a C$^*$-algebra $C_J^*(E)$ and a Jordan $^*$-homomorphism $\beta_E: E\rightarrow C_J^*(E)$ such that $C_J^*(E)$ is the C$^*$-algebra generated by $\beta_E(E)$ and every Jordan $^*$-homomorphism $\pi:E\rightarrow B$, where $B$ is a C$^*$-algebra, extends to a $^*$-homomorphism of $C_J^*(E)$ into $B$. (see \cite[Remark 3.4]{BunFeeTim12})

\begin{lemma}\label{prop:0403162}
If $E$ is a JC$^*$-algebra, then $C_J^*(E)$ is $^*$-isomorphic to $C^*(E)$.
\end{lemma}
\begin{proof}
By definition of $C_J^*(E)$, there exists a $^*$-homomorphism $\widetilde\alpha_E:C_J^*(E)\rightarrow C^*(E)$ such that $\widetilde\alpha_E\circ\beta_E=\alpha_E$.  
By definition of $C^*(E)$, there exists a $^*$-homomorphism $\widetilde\beta_E:C^*(E)\rightarrow C_J^*(E)$ such that $\widetilde\beta_E\circ\alpha_E=\beta_E$.  

By definition of $C_J^*(E)$, there exists a $^*$-homomorphism $(\widetilde\beta_E\circ\alpha_E)^{\widetilde\  }:C_J^*(E)\rightarrow C_J^*(E)$ 
such that $(\tilde\beta_E\circ\alpha_E)^{\widetilde\ }\circ \beta_E=\tilde\beta_E\circ\alpha_E$.  
By definition of $C^*(E)$, there exists a $^*$-homomorphism $(\widetilde\alpha_E\circ\beta_E)^{\widetilde\  }:C^*(E)\rightarrow C^*(E)$ 
such that $(\tilde\alpha_E\circ\beta_E)^{\widetilde\ }\circ \alpha_E=\tilde\alpha_E\circ\beta_E$.  

By diagram chasing $(\tilde\alpha_E\circ\beta_E)^{\widetilde\ }=\tilde\alpha_E\circ\tilde\beta_E$ and $(\tilde\beta_E\circ\alpha_E)^{\widetilde\ }=\tilde\beta_E\circ\tilde\alpha_E$. (It is enough to check this on the generating sets $\alpha_E(E)$ and $\beta_E(E)$.)  It follows that 
$\tilde\alpha_E\circ\tilde\beta_E=\hbox{id}_{C^*(E)}$ and $\tilde\beta_E\circ\tilde\alpha_E=\hbox{id}_{C_J^*(E)}$ so that $\tilde\alpha_E$ is a $^*$-isomorphism with inverse $\tilde\beta_E$.
\end{proof}

\begin{theorem}\label{lem:0402161} If $N$ is a continuous von Neumann algebra, then 
$$T^*(H(N,\beta))=N.$$
\end{theorem}

\begin{proof}
Let $E=H(N,\beta)$.
By \cite[Proposition 3.7]{BunFeeTim12}, $T^*(E)=C_J^*(E)$.  By Lemma~\ref{prop:0403162}, $C_J^*(E)\simeq C^*(E)$.  By \cite[Proposition 2.2]{BunTimJLMS13}, $E$ is universally reversible.
In Theorem~\ref{thm:0403161}, let  $A=E_{sa}$, $B=N$, $\alpha=\beta$ and $\theta (x)=x$ for $x\in A$.  By \cite[Corollary 2.9]{Gasemyr90}, $N$ is the C$^*$-algebra generated by $\theta (A)$, so that Theorem~\ref{thm:0403161} applies to finish the proof.
\end{proof}

\begin{remark}
\cite[Corollary 2.9]{Gasemyr90}, which was used in the proof of Theorem~\ref{lem:0402161}, is a corollary to \cite[Theorem 2.8]{Gasemyr90}, which states that
if $N$ is a von Neumann algebra admitting a $^*$-antiautomorphism $\alpha$ and if $H(N,\alpha)_{sa}$ has no type $I_1$ part, then $N$ is generated as a von Neumann algebra by 
$H(N,\alpha)_{sa}$.   The author of \cite{Gasemyr90} was apparently unaware that \cite[Corollary 2.9]{Gasemyr90} was proved in the case of a continuous factor by Ayupov in 1985 \cite{A7}, and the theorem in this case appeared as Theorem 1.5.2 in the book \cite{AyuRakUsm97} in 1997. The authors of \cite{AyuRakUsm97} state on page 70: ``Theorem 1.5.2 was obtained by Ayupov in
\cite{A3,A6,A7,A8}. Different versions were given by Stormer \cite{Stormer67,Stormer68} and also in the monograph \cite[Chapter 7]{HanOlsSto}.''
\end{remark}

\section{Structure of W*-TROs via JC*-triples}

Now suppose that $X$ is a W*-TRO, and consider the space $X$ with the JC*-triple structure  given by $\{xyz\}=(xy^*z+zy^*x)/2$, so that $X$ becomes a JBW*-triple.  As noted in (\ref{eq:0718161}),    there is a surjective linear isometry
\begin{equation}\label{eq:0718163}
X_\mapsto \oplus_\alpha L^\infty(\Omega_\alpha,C_\alpha)\oplus pM\oplus H(N,\beta),
\end{equation}
where each $C_\alpha$ is a Cartan factor, $M$ and $N$ are continuous von Neumann algebras, $p$ is a projection in $M$,  $\beta$ is a *-antiautomorphism of $N$ of order 2 with fixed points $H(N,\beta)$.


The author acknowledges that  in the following proposition, (a) is only a mild improvement of the results of
Theorem~\ref{thm:0802161}, and Corollary~\ref{cor:0802161} was proved by Ruan \cite{Ruan04} without the separability assumption.  However, the approach is different and has promise for future research (see section~\ref{sect5}).

\begin{proposition}\label{thm:0608161} Let $V$ be a W$^*$-TRO. 

(a)  If $V$ has no type I part, then it is TRO-isomorphic to $eA\oplus Af$,
where $A$ is a continuous von Neumann algebra. 

(b) If $V$ acts on a separable Hilbert space, then it is TRO-isomorphic to 
\[
\oplus_\alpha L^\infty(\Omega_\alpha, B(H_\alpha,K_\alpha))\oplus eA\oplus Af
\]
where $A$ is a continuous von Neumann algebra. 
\end{proposition}
\begin{proof}
For any W$^*$-TRO, by (\ref{eq:0718163}),  write $V=V_1\oplus V_2\oplus V_3$, 
where $V_i$ are weak*-closed orthogonal triple ideals of $V$ with $V_1$ triple isomorphic to a JBW$^*$-triple $\oplus_\alpha L^\infty(\Omega_\alpha,C_\alpha)$ of type $I$, $V_2$ triple isomorphic to a right ideal $pM$ in a continuous von Neumann algebra $M$, and $V_3$  triple isomorphic to $H(N,\beta)$ for some continuous von Neumann algebra $N$ admitting a $^*$-antiautomorphism $\beta$ of order 2.

Since the triple ideals coincide with the TRO ideals (see \cite[Lemma 2.1]{BunFeeTim12}),  in particular  each $V_i$ is a sub-W$^*$-TRO of $V$.

Consider first $V_2$.  By Theorem~\ref{thm:0802161}(a), $V_2$ is TRO-isomorphic to $eA\oplus Af$, for some von Neumann algebra $A$. 
  In particular, $V_2$ is triple isomorphic to $eA\oplus f^tA^t=(e\oplus f^t)(A\oplus A^t)$ and to $pM$, so by \cite{Horn88},   $A\oplus A^t$ has the same type as $M$.  It follows that $A$ is a continuous von Neumann algebra. 
  
Next it is shown that $V_3=0$.  $V_3$ is triple isomorphic to $H(N,\beta)$ and TRO-isomorphic to 
 $eA\oplus Af$, for a von Neumann algebra $A$. 
 Thus the continuous JBW$^*$-triple $H(N,\beta)$ of hermitian type is triple isomorphic to the JBW$^*$-triple $(e\oplus f^t)(A\oplus A^t)$, which is necessarily continuous and hence of associative type.  By the uniqueness of the representation theorem for  continuous JBW$^*$-triples (\cite[Section 4]{Horn88}), $H(N,\beta)=0$.   (For  alternate proofs of the descriptions of $V_2$ and $V_3$ just given, using  techniques from the theories of Jordan triples and universal enveloping TROs, see section~\ref{sect5}.)

 Finally, consider $V_1$.  It will be shown that if $V$ has no type I part, then $V_1=0$, which would prove (a); and if $V$ acts on a separable Hilbert space, then $V_1$ is of the form $\oplus_\alpha L^\infty(\Omega_\alpha, B(H_\alpha,K_\alpha))$, up to TRO-isomorphism, which would prove (b) and complete the proof of the theorem.
 
 There are weak*-closed TRO ideals  $V_\alpha$ such that $V_1=\oplus_\alpha V_\alpha$ with $V_\alpha$ triple isomorphic to $L^\infty(\Omega_\alpha,C_\alpha)$ provided that $V_\alpha\ne 0$, which is assumed henceforth.  It is shown in \cite[Lemma 2.4 and Proof of Theorem 1.1]{IsiSta03} that no Cartan factor of type 2,3,4,5,6 can be isometric to a TRO. It follows easily that $L^\infty(\Omega_\alpha,C_\alpha)$ cannot be isometric to a TRO unless $C_\alpha$ is  a Cartan factor of type 1.  Therefore each $C_\alpha$ is a Cartan factor of type 1, and therefore 
$ V_\alpha$ is either zero, or triple isomorphic to  $L^\infty(\Omega_\alpha,B(H_\alpha,K_\alpha))$ for suitable Hilbert spaces $H_\alpha$ and $K_\alpha$.  

Next consider  $V_\alpha$ for a fixed $\alpha$.
To simplify notation let $U$ denote $V_\alpha$ and $W$ denote $L^\infty(\Omega, B(H,K))$. 
By \cite[Theorem 4.9]{BunTim13}, 
\begin{eqnarray*}T^*(W)&=&L^\infty(\Omega, B(H,K)\oplus B(H,K)^t)\\
&=&L^\infty(\Omega,B(H,K))\oplus L^\infty(\Omega,B(H,K)^t)
\end{eqnarray*}
 and $\alpha_W(x)(\omega)=x(\omega)\oplus x(\omega)^t$, for $x\in T^*(W)$ and $\omega\in\Omega$.
 
 By Theorem~\ref{thm:0802161}(a), $U$ is TRO-isomorphic to $eA\oplus Af$, for some von Neumann algebra $A$. 
 Since $T^*(U)$ is TRO-isomorphic to $T^*(W)$,  by Theorem~\ref{prop:0416161},
 \begin{equation}\label{eq:0721161}
 eA\oplus A^te^t\oplus Af\oplus f^tA^t\stackrel{TRO}{\simeq}L^\infty(\Omega,B(H,K))\oplus L^\infty(\Omega,B(H,K)^t).
 \end{equation}
 The right side of (\ref{eq:0721161}) is a JBW$^*$-triple of type I and thus by \cite[Theorem 5.2]{ChuNeaRus04} or \cite[Theorem 4.2]{BunPer02},  $eA$ is a JBW$^*$-triple of type I, which implies that $A$ is a von Neumann algebra of type I. 
 
Summarizing up to this point, $V$ is arbitrary, and  $V=V_1\oplus V_2+V_3$, where 
\begin{equation}\label{eq:0803161}
V_1\stackrel{TRO}{\simeq}\oplus_\alpha e_\alpha A_\alpha\oplus A_\alpha f_\alpha,
\end{equation}
 $$V_2\stackrel{TRO}{\simeq}eA\oplus Af,\quad V_3=0,$$ where each $A_\alpha $ is a von Neumann algebra of type I, and $A$ is a continuous von Neumann algebra.

   Now suppose that $V$ has no type I part. Then $M(V)$ has no type I part and the same holds for $M(V_\alpha)$.  But $M(V_\alpha)$ is *-isomorphic to  $e_\alpha A_\alpha e_\alpha \oplus c(f_\alpha)A_\alpha$, which is a von Neumann algebra of type I, hence $V_\alpha=0$.  But it was assumed that $V_\alpha\ne 0$ so this contradiction shows that $V_1=0$ and (a) is proved.\smallskip
 
 To prove (b) consider again $V_1$, and focus on a component on the right side of (\ref{eq:0803161}) for a fixed $\alpha$, which is denoted, again for notation's sake, by $eB\oplus Bf$ where $B$ is a von Neumann algebra of type I.  Write $B=\oplus_{\gamma\in \Gamma} L^\infty(\Sigma_\gamma,B(H_\gamma))$, $e=\oplus_\gamma e_\gamma$, and $f=\oplus_\gamma f_\gamma$ so that
 
 $$eB=\oplus_{\gamma\in \Gamma} e_\gamma L^\infty(\Sigma_\gamma,B(H_\gamma)),$$

 $$ Bf=\oplus_{\gamma\in \Gamma}  L^\infty(\Sigma_\gamma,B(H_\gamma))f_\gamma.$$
 
The reduction theory of von Neumann algebras (\cite[Part II]{Dixmier57}) will now be used  to conclude this proof, so assume that $B$ acts on a separable Hilbert space.  For a fixed $\gamma\in\Gamma$,  
 
 $$L^\infty(\Sigma_\gamma,B(H_\gamma))=\int_{\Sigma_\gamma}^\oplus B(H_\gamma)
 \, d\mu_\gamma(\sigma_\gamma),$$

 $$
 L^2(\Sigma_\gamma,H_\gamma)=\int_{\Sigma_\gamma}^\oplus H_\gamma
 \, d\mu_\gamma(\sigma_\gamma),
 $$
 
 \begin{center}
 $
 B=\sum_{\gamma\in\Gamma}^\oplus \int_{\Sigma_\gamma}^\oplus B(H_\gamma)
 \, d\mu_\gamma(\sigma_\gamma),
 $
  \end{center}

 $$
 e_\gamma=\int_{\Sigma_\gamma}^\oplus e_\gamma(\sigma_\gamma)
 \, d\mu_\gamma(\sigma_\gamma),
  $$
 
 and
 
  \begin{center}
 $
 eB=\sum_{\gamma\in\Gamma}^\oplus \int_{\Sigma_\gamma}^\oplus e_\gamma(\sigma_\gamma)B(H_\gamma)
 \, d\mu_\gamma(\sigma_\gamma).
 $
  \end{center}

 For notation's sake, for a fixed $\gamma\in\Gamma$, let $\sigma=\sigma_\gamma$, $\mu=\mu_\gamma$, $e=e_\gamma$, $\Sigma=\Sigma_\gamma$, $H=H_\gamma$, and suppose $H$ is a separable Hilbert space. For each $n\le\aleph_0$, let $\Sigma_n=\{\sigma\in\Sigma:e(\sigma)\hbox{ has rank }n\}$, $e_n=e|_{\Sigma_n}$, and let $K_n$ be a Hilbert space of dimension $n$.  Then
 
  \begin{center}
 $
 \int_\Sigma^\oplus e(\sigma)B(H)\, d\mu(\sigma)=\sum_{n\le \aleph_0}^\oplus \int_{\Sigma_n}^\oplus e_n(\sigma)B(H)
 \, d\mu(\sigma).
 $
  \end{center}

For each $\sigma\in \Sigma_n$,  let $G_\sigma=\{\hbox{all unitaries }U:e_n(\sigma)H\rightarrow K_n\},$
let $G=\cup_{\sigma\in \Sigma_n}G_\sigma$, 
and then set 
\[
E=\{(\sigma,U)\in\Sigma_n\times G: U\in G_\sigma\}.
\]
By the measurable selection theorem \cite[Appendix V]{Dixmier57},  there exists a $\mu$-measurable subset $\Sigma_n'\subset \Sigma_n$ of full measure and a $\mu$-measurable mapping $\eta$ of $\Sigma_n'$ into $G$, such that $\eta(\sigma)\in G_\sigma$ for every $\sigma\in \Sigma_n'$.\smallskip

It is easy to verify that for each $\sigma\in\Sigma_n'$, $T_{n,\sigma}:e_n(\sigma)x\mapsto \eta(\sigma)e_n(\sigma)x$ is a 
TRO-isomorphism of $e_n(\sigma)B(H)$ onto $B(H,K_n)$  and that 
$\{T_{n,\sigma}:\sigma\in\Sigma_n'\}$ is a $\mu$-measurable field of TRO-isomorphisms.

Hence $\int_{\Sigma_n}^\oplus T_{n,\sigma}\, d\mu(\sigma)$ is a TRO-isomorphism of 
$\int_{\Sigma_n}^\oplus e_n(\sigma) B(H)\, d\mu(\sigma)$ onto 
$\int_{\Sigma_n}^\oplus B(H,K_n)\, d\mu(\sigma)$, that is
 $$
\int_{\Sigma_n}^\oplus e_n(\sigma) B(H)\, d\mu(\sigma)\stackrel{TRO}{\simeq}
L^\infty(\Sigma_n, B(H,K_n)).$$
Going back to the earlier notation, since

\begin{center}
 $
 eB=\sum_{\gamma\in\Gamma}^\oplus \int_{\Sigma_\gamma}^\oplus e_\gamma(\sigma_\gamma)B(H_\gamma)
 \, d\mu_\gamma(\sigma_\gamma).
 $
  \end{center}
it follows  that
 \begin{center}
 $
 eB
  \stackrel{TRO}{\simeq}\sum_{\gamma\in\Gamma}^\oplus \sum_{n\le\aleph_0}L^\infty(\Sigma_{\gamma,n}, B(H_\gamma,K_n)).
 $
  \end{center}
By the same arguments, it is clear that also 

 \begin{center}
 $
 Bf
  \stackrel{TRO}{\simeq}\sum_{\gamma\in\Gamma'}^\oplus \sum_{n\le\aleph_0}L^\infty(\Sigma_{\gamma,n}', B(K_n,H_\gamma')).
 $
  \end{center}
Recalling that $B$ was one of the $A_\alpha$ in (\ref{eq:0803161}), this completes the proof of (b).  \end{proof}

\begin{corollary}[Ruan]\label{cor:0802161}
A W$^*$-TRO of type I, acting on a separable Hilbert space,  is TRO-isomorphic to $\oplus_\alpha L^\infty(\Omega_\alpha, B(H_\alpha,K_\alpha))$.
\end{corollary}
\begin{corollary}\label{cor:0718161}
A W$^*$-TRO of type $II_{1,1}$ is TRO-isomorphic to $eA\oplus Af$, where $e,f$ are centrally orthogonal projections in a von Neumann algebra $A$ of type $II_1$.
\end{corollary}

\section{Alternate proofs}\label{sect5}
Presented here are alternate approaches to the proofs of the assertions concerning $V_2$ and $V_3$ in the proof of Proposition~\ref{thm:0608161}(a), along the lines of the proof of the assertion concerning $V_1$.  The purpose for doing this is that, despite the fact that the proofs are longer,  they illustrate the power of the techniques used from \cite{BunFeeTim12} and \cite{BunTimJLMS13}. 

Consider first $V_2$.  In what follows, it is assumed that $V$ has no type I part.  Recall that $V_2$ is triple isomorphic to a right ideal $pM$ in a continuous von Neumann algebra $M$.  For notation's sake,  denote $V_2$ by $V$ and $pM$ by $W$.
By Theorem~\ref{prop:0416161}, $T^*(W)=W\oplus W^t$ and $\alpha_W(x)=x\oplus x^t$.
By \cite[Proposition 3.9]{BunTimJLMS13}),
and \cite[Theorem 4.11]{BunTimJLMS13})
$W$ does not admit a triple homomorphism onto a Hilbert space of dimension greater than 2, and therefore the same holds for $V$. 

Next it is shown that $W$ does not admit a triple homomorphism onto $\IC$, and it follows that $V$ does not admit a triple homomorphism, and {\it a priori}, a TRO-homomorphism onto $\IC$, thus guaranteeing, by \cite[Theorem 5.4]{BunTimJLMS13}),
 that $T^*(V)=V\oplus V^t$ and $\alpha_V(x)=x\oplus x^t$.

Suppose then, that $f:pM\rightarrow\IC$ is a nonzero triple homomorphism, that is, for $x,y,z\in M$,
\begin{equation}\label{eq:0616161}
f\{px,py,pz\}=f\left(\frac{pxy^*pz+pzy^*px}{2}\right)=f(px)\overline{f(py)}f(pz).
\end{equation}
Putting $x=y=z=1$ in (\ref{eq:0616161}) yields $f(p)=|f(p)|^2f(p)$, so either $f(p)=0$ or
$|f(p)|=1$.  Suppose $f(p)=0$. Then setting $y=1$ in (\ref{eq:0616161}) yields
\[
f(pxpz+pzpx)=0,\quad (x,z\in M)
\]
and setting $z=1$ in (\ref{eq:0616161}) yields
\[
f(pxy^*p+py^*px)=0,\quad (x,y\in M),
\]
which implies 
\[
f(pxp+px)=0\quad (x\in M).
\]
Thus 
\[
0=f(pxy^*p)+f(py^*px)=-f(pxy^*)-f(pxy^*p)
\]
and in particular 
\[
0=f(py^*p)+f(py^*p)=-f(py^*)-f(py^*p)
\]
so that
$f(py^*)=0$ for $y\in M$, that is, $f=0$.

Assume now without loss of generality, that $f(p)=1$.
Writing $(pxp)(pyp)=(pxp)p^*(pyp)$, then for $x,y\in M$,
\[
f((pxp)\circ(pyp))=f\{pxp,p,pyp\}=f(pxp)f(pyp)
\]
so that $f$ is a Jordan $^*$-homomorphism of $pMp$ onto $\IC$. It follows  that $f$ is a normal dispersion-free state on a continuous von
Neumann algebra, and hence must be zero (see the proof of Lemma~\ref{lem:0416161}).

Thus $T^*(V)=V\oplus V^t$, $\alpha_V(x)=x\oplus x^t$ and  there is a weak*-continuous TRO-isomorphism  of $T^*(V)$ onto $T^*(W)$, by
\cite[Proposition 2.4]{EffOzaRua01}.  Thus $V$ is TRO-isomorphic to a weak*-closed ideal $I$ in $W\oplus W^t$.  Writing $I=(I\cap W)\oplus (I\cap W^t)$, then $I\cap W$ is a weak*-closed ideal in $W$, let's call it $I_1$,  and $I\cap W^t$ is a weak*-closed ideal in $W^t$, let's call it $I_2$.  As noted in \cite{Horn88}, there are projections $p_1\le p,p_2\le p^t$  such that $I_1=p_1M$ and $I_2=M^tp_2$.

More precisely, 
\begin{equation}\label{eq:0711161}
I=I_1\oplus I_2=(p_1\oplus 0)(M\oplus M^t)\oplus (M\oplus M^t)(0\oplus p_2)=eA\oplus Af,
\end{equation}
where $A=M\oplus M^t$ is a continuous von Neumann algebra, $e=p_1\oplus 0$ and $f=0\oplus p_2$.

With regard to Corollary~\ref{cor:0718161}, suppose now that $V$ is of type $II_{1,1}$. It will be shown  that $A$ can be chosen to be of type $II_1$. Since
\[
R_V\stackrel{^*-iso.}{\simeq}R_{I_1}\oplus R_{I_2} =\left[\begin{matrix}
eAe  &  eA \\
Ae&c(e)A
\end{matrix}\right]\oplus
\left[\begin{matrix}
c(f)A  &  Af \\
fA&fAf
\end{matrix}\right],
\]
it follows  that $c(f)A$ and $c(e)A$ are each of type $II_1$.


Since $p_1M=p_1(c(p_1)M)$ and $M^tp_2=(M^tc(p_2))p_2$, 
if  $A=M\oplus M^t$ is replaced by $\tilde A=c(p_1)M\oplus c(p_2)M^t$, then $\tilde A$ is a continous von Neumann algebra, $eA\oplus Af=e\tilde A\oplus \tilde A f$, and 
\[
R_V\stackrel{^*-iso.}{\simeq}\left[\begin{matrix}
e\tilde Ae  &  e\tilde A \\
\tilde Ae&\tilde A
\end{matrix}\right]\oplus
\left[\begin{matrix}
\tilde A  &  \tilde Af \\
f\tilde A&f\tilde Af
\end{matrix}\right],
\]
so that $\tilde A$ is of type $II_1$.\smallskip



Consider next  $V_3$.  $V_3$ is  triple isomorphic to $H(N,\beta)$ for some continuous von Neumann algebra $N$ which admits a $^*$-anti-automorphism $\beta$ of order 2.   For notation's sake,  denote $V_3$ by $V$ and $H(N,\beta)$ by $W$.

Note first  that $V$ is a universally reversible TRO.  This follows by the same arguments which were used in the discussion of $V_2$ in this subsection.  Indeed, by \cite[Proposition 2.2]{BunTimJLMS13} and the paragraph preceding it, 
$W$ is a universally reversible JC$^*$-triple, and therefore so is $V$.
As before,  $V$ does not admit a triple homomorphism onto a Hilbert space of dimension different from 2.  

On the other hand, $V$ has no nonzero TRO-homomorphism onto $\IC$, since such a homomorphism would extend to a $^*$-homomorphism of the linking von Neumann algebra $R_V$ of $V$ onto $M_2(\IC)$, whose restriction $\rho$ to the upper left corner of $R_V$ would be a dispersion-free state on a continuous von Neumann algebra.  It is easily seen that 
$\rho$ is completely additive on projections, hence normal and hence cannot exist (see the proof of Lemma~\ref{lem:0416161}).

So $T^*(V)=V\oplus V^t$, $\alpha_V(x)=x\oplus x^t$, and $V\oplus V^t$ is TRO-isomorphic to $T^*(W)=N$, by Theorem~\ref{lem:0402161}.
By
\cite[Proposition 2.4]{EffOzaRua01}, the TRO-isomorphism  is weak*-continuous.  Hence the weak*-closed TRO ideal $V$  in $V\oplus V^t$ is mapped onto a weak*-closed TRO ideal in $N$, which is necessarily a two-sided ideal in $N$, say $zN$ for some central projection $z$ in $N$. From (\ref{eq:0711161}) it follows that
$
V_2\oplus V_3
$
is TRO-isomorphic to
\[
 [(e\oplus 0)( A\oplus N)]\oplus [( A\oplus N)(f\oplus 0)\oplus [(0\oplus z)( A\oplus N)],
\]
so that $V_2\oplus V_3$ is the direct sum of a weakly closed left ideal and a weakly closed right ideal in a continuous von Neumann algebra, which is tantamount to  proving that $V_3=0$.

This last argument shows that (a) implies (b) in Theorem~\ref{thm:0802161}.







\begin{bibdiv}
\begin{biblist}

\bib{AlfSchbook}{book}{
   author={Erik M. Alfsen},
   author={Frederic W. Shultz},
   title={Geometry of State Spaces of Operator Algebras},
   series={Mathematics Theory and Applications},
   publisher={Birkha\"user Boston Basel Berlin},
    date={2003},
     pages={xiii+467 pp},
  }


\bib{A3}{article}{
   author={Ayupov, Shavkat},
   title={On the construction of Jordan algebras of self-adjoint operators},
   journal={Soviet Math.\ Dokl.},
   volume={26},
   date={1982},
   number={3},
   pages={623--626},
  }

\bib{A6}{article}{
   author={Ayupov, Shavkat},
   title={On existence of Jordan algebras of self-adjoint operators of a given type},
   journal={Siberian Math.\ J.},
   volume={25},
   date={1984},
   number={1-6},
   pages={689--693},
  }

\bib{A7}{article}{
   author={Ayupov, Shavkat},
   title={JW-factors and anti-automorphisms of von Neumann algebras},
   journal={Math. USSR-Investiya},
   volume={26},
   date={1986},
   pages={201--209},
  }

\bib{A8}{article}{
   author={Ayupov, Shavkat},
   title={Jordan operator algebras},
   journal={J. of Soviet Mathematics},
   volume={37},
   date={1987},
   number={6},
   pages={1422--1448},
  }

\bib{AyuRakUsm97}{book}{
   author={Ayupov, Shavkat},
   author={Rakhimov, Abdugafur},
   author={Usmanov, Shukhrat },
   title={Jordan, real and Lie structures in operator algebras},
   series={Mathematics and its Applications},
   volume={418},
   publisher={Kluwer Academic Publishers Group, Dordrecht},
   date={1997},
   pages={x+225 pp},
  }


 \bib{BohWer14}{article}{
   author={Bohle, Dennis},
    author={Werner, Wend},
   title={The universal enveloping ternary ring of operators of a JB*-triple system},
   journal={Proc. Edinb. Math. Soc. (2)},
   volume={57},
    date={2014},
    number={2},
     pages={347--366},
  }


\bib{BunFeeTim12}{article}{
   author={Bunce, Leslie J.},
   author={Feely, Brian},
author={Timoney, Richard M},
   title={Operator space structure of JC*-triples and TROs, I},
   journal={Math. Z.},
   volume={270},
   number={3-4},
   date={2012},
   pages={961--982},
  }

\bib{BunPer02}{article}{
   author={Bunce, Leslie J.},
author={Peralta, Antonio M.},
   title={Images of contractive projections on operator algebras},
   journal={J.\ Math.\ Anal.\ Appl.},
   volume={272},
   date={2002},
   pages={55--66},
  }

\bib{BunTim13}{article}{
   author={Bunce, Leslie J.},
author={Timoney, Richard M},
   title={On the universal TRO of a JC*-triple, ideals and tensor products},
   journal={Q.\ J.\ Math.},
   volume={64},
   number={2},
   date={2013},
   pages={327--340},
  }

\bib{BunTimJLMS13}{article}{
   author={Bunce, Leslie J.},
author={Timoney, Richard M},
   title={Universally reversible JC*-triples and operator spaces},
   journal={J.\ Lon.\ Math.\ Soc.},
   volume={(2) 88},
   date={2013},
   pages={271--293},
  }

\bib{chu}{book}{
   author={Cho-Ho Chu},
   title={Jordan structures in geometry and analysis},
   series={Cambridge Tracts in Mathematics},
   volume={190},
   publisher={Cambridge University Press, Cambridge},
    date={2012},
     pages={x+261 pp},
  }


\bib{ChuNeaRus04}{article}{
   author={Chu, Cho-Ho},
author={Neal, Matthew},
author={Russo, Bernard},
   title={Normal contractive projections preserve type},
   journal={J.\ Operator Theory},
   volume={(51},
   date={2004},
   pages={281--301},
  }


 \bib{Dixmier53}{article}{
   author={Dixmier, Jacques}, 
   title={Formes lin\'eaires sur an anneau d'op\' erateurs},
   journal={Bull.\ Soc.\ Math.\ France},
   volume={81},
    date={1953},
     pages={9--39},
  }



\bib{Dixmier57}{book}{
   author={Jacques Dixmier},
   title={Von Neumann Algebras},
   volume={27},
   publisher={Elsevier North-Holland},
    date={1981},
     pages={xxxviii+437 pp},
  }


\bib{EffOzaRua01}{article}{
   author={Edward G.\ Effros},
   author={Narutaka Ozawa},
   author={Zhong-Jin Ruan},
   title={On infectivity and nuclearity for operator spaces},
   journal={Duke Math. J.},
   volume={110},
    date={2001},
    number={3},
     pages={489--522},
  }



 \bib{Gasemyr90}{article}{
   author={Gasemyr, Jorund}, 
   title={Involutory antiautomorphisms of von Neumann and C$^*$-algebras},
   journal={Math.\ Scand.},
   volume={67},
    date={1990},
     pages={87--96},
  }



 \bib{HancheOlsen03}{article}{
   author={Hanche-Olsen, Harald}, 
   title={On the structure and tensor products of JC-algebras},
   journal={Canad. J. Math.},
   volume={35},
    date={1983},
    number={6},
     pages={1059--1074},
  }


\bib{HanOlsSto}{book}{
   author={Harald Hanche-Olsen},
   uthor={Erling St\ormer},
   title={Jordan operator algebras},
   series={Pitman Advanced Publishing Program},
   publisher={Pitman},
    date={1984},
     pages={viii+183 pp},
  }

\bib{Horn87}{article}{
   author={G\"unther Horn},
   title={Classification of JBW$^*$-triples of type I},
   journal={Math.\  Z.},
   volume={196},
    date={1987},
    number={2},
     pages={271--291},
  }


\bib{Horn88}{article}{
   author={G\"unther Horn}, 
   author={Erhard  Neher},
   title={Classification of continuous JBW$^\ast$-triples},
   journal={Trans.\ Amer.\
Math.\  Soc.},
   volume={306},
    date={1988},
     pages={553--578},
  }

 
   \bib{IsiSta03}{article}{
   author={Isidro, Jos\'e M.},
   author={Stach\'o, Laszlo L.},
   title={On the Jordan structure of ternary rings of operators},
   journal={Ann.\ Univ.\ Sci.\ Budapest. E\"otv\"os Sect.\ Math.},
   volume={46},
    date={2003,2004},
     pages={149--156}
  }

 \bib{KanedaPJM13}{article}{
   author={Kaneda, Masayoshi}, 
   title={Ideal decampsitions of a ternary ring of operators with predual},
   journal={Pac.\  J.\ Math.},
   volume={266},
    date={2013},
    number={2},
     pages={297--303},
  }





   \bib{NeaRus03}{article}{
   author={Neal, Matthew},
   author={Russo, Bernard},
   title={Contractive projections and operator spaces},
   journal={Trans. Amer. Math. Soc.},
   volume={355},
    date={2003},
    number={6},
     pages={2223--2262}
  }

 \bib{Plymen68}{article}{
   author={Plymen, R.\ J.},
   title={Dispersion-free normal states},
   journal={Il Nuovo Cimento. A},
   volume={LIV},
    date={1968},
    number={4},
     pages={862--870}
  }
  
  
 \bib{Ruan04}{article}{
   author={Ruan, Zhong-Jin},
   title={Type decomposition and the rectangular AFD property for W$^*$-TROs},
   journal={Canad. J. Math.},
   volume={36},
    date={2004},
    number={4},
     pages={843--870}
  }

 \bib{sakai}{book}{
   author={S. Sakai},
   title={C*-algebras and W*-algebras},
   series={Ergebnisse der Mathematik und ihrer Grenzgebiete},
   volume={60},
   publisher={Springer-Verlag, New York Heidelberg Berlin},
    date={1971},
     pages={xii+256 pp},
  }


 \bib{Stormer67}{article}{
   author={St\o rmer, Erling},
   title={On anti-automorphisms of von Neumann algebras},
   journal={Pacific J. Math.},
   volume={21},
    date={1967},
     pages={349--370}
  }

  \bib{Stormer68}{article}{
   author={St\o rmer, Erling},
   title={Irreducible Jordan algebras of self-adjoint operators},
   journal={Trans. Amer. Math. Soc.},
   volume={130},
    date={1968},
     pages={153--166}
  }
\end{biblist}
\end{bibdiv}
\end{document}